\documentclass[reqno,12pt]{amsart} 
\usepackage{vmargin}
\setpapersize{A4}

\usepackage{amsmath} 

     \usepackage{amsfonts}   

\usepackage{amssymb}      

\usepackage{eufrak}      

\usepackage{amscd}      

\usepackage{amsthm}      
   
\usepackage{epsfig}      

\usepackage{amstext}      

\usepackage[all,line,dvips]{xy} 
\CompileMatrices 

\newcommand{\id}{\operatorname{id}}

 \newcommand{\supp}{\operatorname{supp}}

\newcommand{\IM}{\operatorname{Im}}

\newcommand{\EXP}{\operatorname{exp}}

   \theoremstyle{plain}
   \newtheorem{thm}{Theorem}[section]
   \newtheorem{prop}[thm]{Proposition}
   \newtheorem{lemma}[thm]{Lemma}  
   \newtheorem{cor}[thm]{Corollary}
   \theoremstyle{definition}

   \theoremstyle{remark}

\usepackage{graphicx}

   \numberwithin{equation}{section}

        \date{\today}

\title[The $C^*$-algebra of the exponential function]{The
  $C^*$-algebra of the exponential function}
  
\author{Klaus Thomsen}

\date{\today}

\email{matkt@imf.au.dk}
\address{Institut for matematiske fag, Ny Munkegade, 8000 Aarhus C, Denmark}

\begin{document}

\maketitle

\begin{abstract} The complex exponential function $e^z$ is a local
  homeomorphism and gives therefore rise to an \'etale groupoid and a
  $C^*$-algebra. We show that this $C^*$-algebra is simple, purely
  infinite, stable and classifiable by K-theory, and has both K-theory groups
  isomorphic to $\mathbb Z$. The same methods show that the $C^*$-algebra of the
  anti-holomorphic function $\overline{e^z}$ is the stabilisation of the
  Cuntz-algebra $\mathcal O_3$. 

\end{abstract}

\section{Introduction} 
The crossed product of a locally compact
Hausdorff space by a homeomorphism has been generalised to local
homeomorphisms in the work of Renault, Deaconu and
Anantharaman-Delaroche, \cite{Re},\cite{De}, \cite{An}. In many cases
the algebra is both simple and purely infinite, and can be determined
by the use of the Kirchberg-Phillips classification result. The
purpose with the present note is to demonstrate how methods and
results about the iteration of complex holomorphic functions can be
used for this purpose. This will be done by determining the
$C^*$-algebra of an entire holomorphic function $f$ when $f'(z) \neq
0$ and $\# f^{-1}(f(z)) \geq
2$ for all $z \in \mathbb C$, and the Julia set $J(f)$ of $f$ is the
whole complex plane $\mathbb C$. A prominent class of functions with
these properties is the family $\lambda e^z$ where $\lambda >
\frac{1}{e}$. These functions commute with complex conjugation and we
can therefore use the same methods to determine the $C^*$-algebra of the
anti-holomorphic function $\overline{f}$. The results are as stated in
the abstract for $f(z) = e^z$.

\section{The $C^*$-algebra of a local homeomorphism}

\subsection{The definition}
We describe in this section the construction of a $C^*$-algebra from a local
homeomorphism. It was introduced in increasing generality by
J. Renault \cite{Re}, V. Deaconu \cite{De} and Anantharaman-Delaroche
\cite{An}.

Let $X$ be a second countable locally compact Hausdorff space and
$\varphi : X \to X$ a local homeomorphism. Set
$$
\Gamma_{\varphi} = \left\{ (x,k,y) \in X \times \mathbb Z  \times X :
  \ \exists n,m \in \mathbb N, \ k = n -m , \ \varphi^n(x) =
  \varphi^m(y)\right\} .
$$
This is a groupoid with the set of composable pairs being
$$
\Gamma_{\varphi}^{(2)} \ =  \ \left\{\left((x,k,y), (x',k',y')\right) \in \Gamma_{\varphi} \times
  \Gamma_{\varphi} : \ y = x'\right\}.
$$
The multiplication and inversion are given by 
$$
(x,k,y)(y,k',y') = (x,k+k',y') \ \text{and}  \ (x,k,y)^{-1} = (y,-k,x)
.
$$
Note that the unit space of $\Gamma_{\varphi}$ can be identified with
$X$ via the map $x \mapsto (x,0,x)$. Under this identification the
range map $r: \Gamma_{\varphi} \to X$ is the projection $r(x,k,y) = x$
and the source map the projection $s(x,k,y) = y$.

To turn $\Gamma_{\varphi}$ into a locally compact topological groupoid, fix $k \in \mathbb Z$. For each $n \in \mathbb N$ such that
$n+k \geq 0$, set
$$
{\Gamma_{\varphi}}(k,n) = \left\{ \left(x,l, y\right) \in X \times \mathbb
  Z \times X: \ l =k, \ \varphi^{k+n}(x) = \varphi^n(y) \right\} .
$$
This is a closed subset of the topological product $X \times \mathbb Z
\times X$ and hence a locally compact Hausdorff space in the relative
topology.
Since $\varphi$ is locally injective $\Gamma_{\varphi}(k,n)$ is an open subset of
$\Gamma_{\varphi}(k,n+1)$ and hence the union
$$
{\Gamma_{\varphi}}(k) = \bigcup_{n \geq -k} {\Gamma_{\varphi}}(k,n) 
$$
is a locally compact Hausdorff space in the inductive limit topology. The disjoint union
$$
\Gamma_{\varphi} = \bigcup_{k \in \mathbb Z} {\Gamma_{\varphi}}(k)
$$
is then a locally compact Hausdorff space in the topology where each
${\Gamma_{\varphi}}(k)$ is an open and closed set. In fact, as is easily verified, $\Gamma_{\varphi}$ is a locally
compact groupoid in the sense of \cite{Re}, i.e. the groupoid
operations are all continuous, and an \'etale groupoid in the sense
that the range and source maps are local homeomorphisms.

To obtain a
$C^*$-algebra, consider the space $C_c\left(\Gamma_{\varphi}\right)$ of
continuous compactly supported functions on $\Gamma_{\varphi}$. They form
a $*$-algebra with respect to the convolution-like product
$$
f \star g (x,k,y) = \sum_{z,n+ m = k} f(x,n,z)g(z,m,y)
$$
and the involution
$$
f^*(x,k,y) = \overline{f(y,-k,x)} .
$$
To obtain a $C^*$-algebra, let $x \in X$ and consider the
Hilbert space $l^2\left(s^{-1}(x)\right)$ of square summable functions on $s^{-1}(x) = \left\{ (x',k,y')
    \in \Gamma_{\varphi} : \ y' = x \right\}$ which carries a
  representation $\pi_x$ of the $*$-algebra $C_c\left(\Gamma_{\varphi}\right)$
  defined such that
\begin{equation*}\label{pirep}
\left(\pi_x(f)\psi\right)(x',k, x) = \sum_{z, n+m = k}
f(x',n,z)\psi(z,m,x)  
\end{equation*}
when $\psi  \in l^2\left(s^{-1}(x)\right)$. One can then define a
$C^*$-algebra $C^*_r\left(\Gamma_{\varphi}\right)$ as the completion
of $C_c\left(\Gamma_{\varphi}\right)$ with respect to the norm
$$
\left\|f\right\| = \sup_{x \in X} \left\|\pi_x(f)\right\| .
$$
Since we assume that $X$ is second countable it follows that
$C^*_r\left(\Gamma_{\varphi}\right)$ is separable. 
Note that this $C^*$-algebra can be constructed from any locally
compact \'etale groupoid $\Gamma$ in the place of
$\Gamma_{\varphi}$, see e.g. \cite{Re}, \cite{An}. Note also that $C^*_r\left(\Gamma_{\varphi}\right)$
is the classical crossed product $C_0(X) \times_{\varphi}
\mathbb Z$ when $\varphi$ is a homeomorphism.

\section{The generalised Pimsner-Voiculescu exact sequence}

There is a six-terms exact sequence which can be used to calculate the
$K$-theory of $C^*_r\left(\Gamma_{\varphi}\right)$. It was obtained
from the work of Pimsner, \cite{Pi}, by Deaconu and Muhly in a
slightly different setting in \cite{DM}. In particular, Deaconu and
Muhly require $\varphi$ to be surjective and essentially free, but thanks
to the work of Katsura in \cite{Ka} we can now establish it for
arbitrary local homeomorphisms. This generalisation will be important
here because $e^z$ is not surjective.

Consider the set
$$
\Gamma_{\varphi}(1,0) = \left\{ \left(x,1,y\right) \in
  \Gamma_{\varphi}(1) : \ y =
  \varphi(x) \right\} ,
$$   
which is an open subset of $\Gamma_{\varphi}(1)$ and hence of $\Gamma_{\varphi}$. Set $E_0 = C_c\left(\Gamma_{\varphi}(1,0)\right)$. Note that $f^*g \in C_c(\varphi(X))
  \subseteq C_c\left(\Gamma_{\varphi}\right)$ when $f,g \in E_0$. In fact
$$
f^*g(x,k,y) = \begin{cases} 0, & \ k \neq 0 \ \vee \ x \neq y \ \vee \
  x \notin \varphi(X) \\ \sum_{z \in
    \varphi^{-1}(x)} \overline{f(z,1,x)}g(z,1,x), & \ k = 0 \ \wedge \ x =
  y \in \varphi(X). \end{cases}
$$
It follows
  that the closure $E$ of $E_0$ in $C^*_r\left(\Gamma_{\varphi}\right)$ is a
  Hilbert $C_0(X)$-module with an $C_0(X)$-valued inner product $\left<
  \cdot, \ \cdot \right>$ defined such that
$\left< f,g\right> = f^*g, \  f,g \in E$. Since $C_0(X) E \subseteq E$
we can consider $E$ as a
$C^*$-correspondence over $C_0(X)$ in the obvious way, cf. Definition
1.3 of \cite{Ka}. Let $\iota :
C_0(X) \to C^*_r\left(\Gamma_{\varphi}\right)$ and $t : E \to
C^*_r\left(\Gamma_{\varphi}\right)$ be the inclusion maps. Then $(\iota,t)$
is an injective representation of the $C^*$-correspondence $E$ in the sense of
Katsura, cf. Definitions 2.1 and 2.2 of \cite{Ka}.

Let $\mathbb K(E)$ be the $C^*$-algebra of adjointable operators on
$E$ generated by the elementary operators $\Theta_{f,g}, f,g \in E$,
where $\Theta_{f,g}(k) = f \left<g,k\right>$.

\begin{lemma}\label{rep} $C_0(X) \subseteq \mathbb K(E)$.
\end{lemma}
\begin{proof} Since $\Theta_{f,g}(k) = fg^*k$ when $f,g,k \in E$ it
  suffices to show that the elements of $C_0(X)$ of the form $fg^*$
  for some $f,g \in E$ span a dense subspace of $C_0(X)$. Let $U$ be
  an open subset of $X$ where $\varphi$ is injective and consider a
  non-negative function $h \in C_c(X)$ supported in $U$. Then
$$
W = \left\{ (x,1,\varphi(x)) : \ x \in U \right\}
$$
is an open subset of $\Gamma_{\varphi}(1,0)$ and we define $f \in E_0$
such that $\supp f \subseteq W$ and $f(x,1,\varphi(x)) = \sqrt{h(x)}, \ x
\in U$. Then $h = \Theta_{f,f}$ and we are done.

\end{proof}

It follows that $(\iota, t)$ is covariant in the sense of \cite{Ka},
cf. Proposition 3.3 and Definition 3.4 of \cite{Ka}, and there
is therefore an associated $*$-homomorphism $\rho : \mathcal O_E \to
C^*_r\left(\Gamma_{\varphi}\right)$.

\begin{prop}\label{cuntzpimsner} $\rho : \mathcal O_E \to
  C^*_r\left(\Gamma_{\varphi}\right)$ is an isomorphism.
\end{prop}
\begin{proof} By construction $C^*_r\left(\Gamma_{\varphi}\right)$ carries an action
$\beta$ by the circle $\mathbb T$ defined such that
$$
\beta_{\lambda}(f)(x,k,y) = \lambda^k f(x,k,y)
$$
when $f \in C_c\left(\Gamma_{\varphi}\right)$. This is the \emph{gauge
  action}. This action ensures that $(\iota,t)$ admits a gauge action in the
  sense of \cite{Ka}. By Theorem 6.4 of \cite{Ka} it suffices
  therefore to show that
  $C^*_r\left(\Gamma_{\varphi}\right)$ is generated, as a $C^*$-algebra, by
  $C_0(X)$ and $E$.

Let $\mathcal A$ be the $*$-subalgebra of
  $C_c\left(\Gamma_{\varphi}\right)$ generated by $C_c(X)$ and
  $E_0$. Let $k \in \mathbb N$. We claim that
  $C_c\left(\Gamma_{\varphi}(k,0)\right) \in \mathcal A$. Since
  $C_c\left(\Gamma_{\varphi}(0,0)\right) = C_c(X)$ and $E_0 =
  C_c\left(\Gamma_{\varphi}(1,0)\right)$ it suffices to prove this
  when $k \geq 2$. To this end
  it suffices, since $C_c(X) \subseteq \mathcal A$, to consider an
  open subset $U$ of $X$ on which $\varphi^k$ is injective and show
  that any non-negative continuous function $h$ compactly supported in
  $\left\{ (x,k,\varphi^k(x)) : \ x \in U\right\}$ is in
  $\mathcal A$. To this end, let $f_j \in E_0$ be supported in
$$
\left\{ (y,1,\varphi(y)) : \ y \in \varphi^{j-1}(U)\right\}
$$
and satisfy that $f_j\left(\varphi^{j-1}(x),1,\varphi^{j}(x)\right) =
h(x,k,\varphi^k(x))^{\frac{1}{k}}$ for all $x \in U$. Then $h =
f_1f_2\cdots f_{k}$ and hence $h \in \mathcal A$.

We
will next prove by induction that $C_c\left(\Gamma_{\varphi}(k,n)\right) \subseteq
\mathcal A$ for all $n \in \mathbb N$. The assertion holds when $n
=0$ as we have just shown, so assume that it holds for $n$. To show that
$C_c\left(\Gamma_{\varphi}(k,n+1)\right) \subseteq \mathcal A$, let $U$ and $V$
be open subsets in $X$ such that $\varphi^n$ is injective on both $U$
and $V$. It suffices to consider a continuous function $h$
compactly supported in $\Gamma_{\varphi}(k,n+1) \cap (U \times \{k\}
\times V)$ and show that $h \in \mathcal A$. Note that $W =\Gamma_{\varphi}(k,n) \cap \left(\varphi(U) \times \{k\} \times
  \varphi(V)\right)$ is open in $\Gamma_{\varphi}(k,n)$ and that we can define
$\tilde{h} : W \to \mathbb C$ such that $\tilde{h}(x,k,y) =
h\left(x',k,y'\right)$, where $x'\in U, y' \in V$ and $\varphi(x') =x, \
  \varphi(y') = y$. Then $\tilde{h}$ is continuous and has
  compact support in $W$; in fact,
  the support is the image of the support, $K$, of $h$ under the continuous
  map $\Gamma_{\varphi}(k,n+1) \ni (x,k,y) \mapsto
  (\varphi(x),k,\varphi(y)) \in \Gamma_{\varphi}(k,n)$. Hence $\tilde{h}
  \in \mathcal A$ by assumption. Note that $r(K)$ is a compact subset
  of $U$ and $s(K)$ a compact subset of $V$. Let $a \in C_c(X)$ be
  supported in $U$ such that $a(x) = 1, \ x \in r(K)$, and $b \in C_c(X)$ be
  supported in $V$ such that $b(x) = 1, \ x \in s(K)$. Define
  $\tilde{a},\tilde{b} \in C_c\left(\Gamma_{\varphi}(1,0)\right) =E_0$
  with supports in $\left\{ (x,1,\varphi(x)) : \ x \in U\right\}$ and
  $\left\{ (x,1,\varphi(x)) : \ x \in V\right\}$, respectively, such that $\tilde{a}(x,1,\varphi(x)) =
  a(x)$ when $x \in U$ and $\tilde{b}(x,1,\varphi(x)) = b(x)$ when $x
  \in V$. Since $h = \tilde{a}\tilde{h}\tilde{b}^*$ we conclude that $h
  \in \mathcal A$. Thus $C_c\left(\Gamma_{\varphi}(k,n)\right) \subseteq \mathcal A$ for all $k
  \geq 0, n \geq 0$. Since $C_c\left(\Gamma_{\varphi}(-k,n)\right)^* =
C_c\left(\Gamma_{\varphi}(n,n-k)\right)$ when $n \geq k \geq 0$, we conclude that $\mathcal A =
C_c\left(\Gamma_{\varphi}\right)$.

\end{proof}

By combining Proposition \ref{cuntzpimsner} and Lemma \ref{rep} with
Theorem 8.6 of \cite{Ka} we obtain the following.

\begin{thm}\label{6terms} (Deaconu and Muhly, \cite{DM}) Let $[E] \in KK\left(C_0(X),C_0(X)\right)$
  be the element represented by the embedding $C_0(X) \subseteq \mathbb
  K(E)$. There is an exact sequence
\begin{equation*}
\begin{xymatrix}{
K_0\left(C_0(X)\right) \ar[r]^-{\id_* - [E]_*} &  K_0\left(C_0(X)\right)
\ar[r]^-{\iota_*} & K_0\left(C^*_r\left(\Gamma_{\varphi}\right)\right) \ar[d]
  \\
 K_1\left(C^*_r\left(\Gamma_{\varphi}\right)\right) \ar[u] &
   K_1\left(C_0(X)\right) \ar[l]^{\iota_*} & K_1\left(C_0(X)\right) \ar[l]^-{\id_* - [E]_*}}
\end{xymatrix}
\end{equation*}
\end{thm}

\section{Simple purely infinite $C^*$-algebras from entire functions
  without critical points in the Julia set}

Throughout this section $f : \mathbb C \to \mathbb C$ is an
  entire function of degree at least 2; i.e. either a polynomial of degree at least 2 or a
  transcendental function. An $n$-periodic point $z \in \mathbb C$ is
  \emph{repelling} when $\left|(f^n)'(z)\right| > 1$. The \emph{Julia
    set} $J(f)$ of $f$ can then be defined as the closure of the repelling
  periodic points. Although this is not the standard definition it
  emphasises one of the properties that will be important here. Others
  are
\begin{enumerate}
\item[i)] $J(f)$ is non-empty and perfect, and
\item[ii)] $J(f)$ is totally $f$-invariant, i.e. $f^{-1}(J(f)) = J(f)$.
\end{enumerate} 
We refer to the survey by Bergweiler, \cite{Be}, for the proof of
these properties.

Let
  $\mathcal E(f)$ denote the set of points $x \in \mathbb C$ such that
  $f^{-1}(x) = \{x\}$. For example, when $f(z) = 2ze^z$ the point $0$
  will be in $\mathcal E(f) \cap J(f)$.

\begin{lemma}\label{1} $\# \mathcal E(f) \leq 1$.
\end{lemma}
\begin{proof} Let $x,y \in \mathcal E(f)$ and assume for a
  contradiction that $x \neq y$. Since $J(f)$ is infinite, $J(f)
  \backslash \{x,y\}$ is not empty. Let $U$ be an open subset of
  $\mathbb C$ such that 
\begin{equation}\label{nonint}
U \cap \left(J(f) \backslash \{x,y\}\right) \neq \emptyset .
\end{equation}
Then $\bigcup_{i=0}^{\infty} f^i\left( U \backslash
  \{x,y\}\right)$ is open, non-empty, $f$-invariant and does not
contain $\{x,y\}$. It follows therefore from Montel's theorem, cf. Theorem 3.7 in
  \cite{Mi}, that $f^n, n \in \mathbb N$, is a normal family when
  restricted to $\bigcup_{i=0}^{\infty} f^i\left( U \backslash
  \{x,y\}\right)$. This contradicts (\ref{nonint}) since $U \backslash
\{x,y\}$ contains a repelling periodic point.
\end{proof}

The set $J(f) \backslash \mathcal E(f)$ is locally compact in the
relative topology inherited from $\mathbb C$ and $f^{-1}\left(J(f)
  \backslash \mathcal E(f)\right) = J(f) \backslash \mathcal E(f)$. If
we now assume that $f'(z) \neq 0$ when $z \in J(f)$, it
follows that the restriction
$$
F :  J(f) \backslash \mathcal E(f) \to J(f) \backslash \mathcal E(f)
$$
of $f$ to $J(f) \backslash \mathcal E(f)$ is a local homeomorphism on
the second countable locally compact Hausdorff space $J(f) \backslash
\mathcal E(f)$. Note that $F$ is not always surjective - it is not when
$f(z) = e^z$.

Following \cite{An} we say that an \'etale groupoid $\Gamma$ with
range map $r$ and source map $s$ is
\emph{essentially free} when the points $x$ of the unit space
$\Gamma^0$ for which the isotropy group $s^{-1}(x) \cap r^{-1}(x)$ is
trivial (i.e. only consists of $\{x\}$) is dense in $\Gamma^0$. For
the groupoid $\Gamma_f$ of a local homeomorphism $f : X \to X$ this occurs if and
only if $\left\{ x \in X : \ f^i(x) = x\right\}$ has empty interior
for all $i \in \mathbb N$.

We say that $\Gamma$ is \emph{minimal} when
there is no open non-empty subset $U$ of $\Gamma^0$, other than $\Gamma^0$, which is $\Gamma$-invariant in the sense
that $r(\gamma) \in U \Leftrightarrow s(\gamma) \in U$ for all $\gamma \in
\Gamma$. This holds for
the groupoid $\Gamma_f$ if and
only if the full orbit $\bigcup_{i,j\in \mathbb N} f^{-i}(f^j(x))$ is
dense in $X$ for all $x \in X$.

Finally, we say that $\Gamma$ is \emph{locally contracting} when every
open non-empty subset of $\Gamma^0$ contains an open non-empty subset
$V$ with the property that there is an open bisection $S$ in $\Gamma$
such that $\overline{V} \subseteq s(S)$ and
$\alpha_{S}^{-1}\left(\overline{V}\right) \subsetneq V$ when $\alpha_S
: r(S) \to s(S)$ is the homeomorphism defined by $S$, cf. Definition
2.1 of \cite{An} (but note that the source map is denoted by $d$ in \cite{An}).

\begin{lemma}\label{julia}  Assume that $f'(z) \neq 0$ for all $z \in
  J(f)$. Then $\Gamma_{F}$ is minimal, essentially free and
  locally contracting in the sense of \cite{An}.
\end{lemma}
\begin{proof} To show that $\Gamma_F$ is essentially free we must show
  that
\begin{equation}\label{eq12}
\left\{ z \in  J(f) \backslash \mathcal E(f) : \ F^i(z) = z \right\}
\end{equation}
has empty interior in $J(f) \backslash \mathcal E(f)$ for all $i
\in \mathbb N$. Assume that $U$ is open in $\mathbb C$ and that $U \cap J(f)
\backslash \mathcal E(f)$ is a non-empty subset of (\ref{eq12}). Since
$J(f)$ is perfect it follows that every point $z_0$ of $U \cap J(f)
\backslash \mathcal E(f)$ is the limit of a sequence from
\begin{equation*}\label{eq13} 
\left\{ z \in \mathbb C: \ f^i(z) = z \right\} \backslash \{z_0\}.
\end{equation*}
Since $f$ is entire it follows that $f^i(z) = z$ for all $z \in
\mathbb C$, contradicting that $J(f) \neq \emptyset$. Hence
$\Gamma_F$ is essentially free.

To show that $\Gamma_F$ is minimal, consider an open subset $U
\subseteq \mathbb C$ such that 
\begin{equation}\label{eq14}
U \cap
  J(f) \backslash \mathcal E(f) \neq \emptyset.
\end{equation} 
Let $W = \bigcup_{i,j \in \mathbb N}
  f^{-j}\left(   f^i(U \backslash \mathcal E(f))\right)$. Since $W$ is
  open (in $\mathbb C$), non-empty, $f$-invariant and has non-trivial intersection with
  $J(f)$ it follows from Montel's theorem, cf. Theorem 3.7 in
  \cite{Mi}, that $\mathbb C \backslash W$ contains at most one
  element. Note that this element must be in $\mathcal E(f)$ because
  $W$ and hence also $\mathbb C \backslash W$ is totally $f$-invariant. It
  follows therefore that $W \cap J(f) \backslash
  \mathcal E(f) = J(f) \backslash
  \mathcal E(f)$. Hence
$$
\bigcup_{i,j \in \mathbb N}
  f^{-j}\left(   f^i(U \cap J(f) \backslash \mathcal E(f))\right) = W
  \cap J(f)  = J(f) \backslash
  \mathcal E(f) ,
$$
proving that $\Gamma_F$ is minimal. 

To show that $\Gamma_F$ is locally contracting consider an open subset
$U$ of $\mathbb C$ such that $U \cap J(f) \backslash \mathcal E(f)
\neq \emptyset$. There is then a repelling periodic point $z_0 \in U
\cap J(f) \backslash \mathcal E(f)$. There is therefore an $n \in
\mathbb N$, a positive number $\kappa > 1$ and an open neighbourhood $W
\subseteq U$ of $z_0$ such that $f^n(z_0) =
z_0$, $f^n$ is injective on $W $
and\begin{equation}\label{eu1}
\left|f^n(y) -z_0\right| \geq \kappa |y -z_0|
\end{equation}  
for all $y \in W$. Let $\delta_0 > 0$ be so
small that 
\begin{equation}\label{eu2}
 \left\{y \in \mathbb C : |y - z_0| \leq \delta_0 \right\} \subseteq
f^n(W) \cap W .
\end{equation} 
$z_0$ is not isolated in $J(f) \backslash \mathcal E(f)$ since $J(f)$
is perfect. There is therefore an element $z_1\in J(f) \backslash
\mathcal E(f)$ such that $0 < \left|z_1-z_0\right| < \delta_0$. Choose
$\delta$ strictly between $\left| z_1 -z_0\right|$ and $\delta_0$ such that 
\begin{equation}\label{cru}
\kappa
\left|z_1-z_0\right| > \delta.
\end{equation}
Set $V_0 = \left\{y \in \mathbb C : |y -z_0| <  \delta \right\}$. 
Then 
\begin{equation}\label{eu102}
\overline{V_0} \cap J(f) \backslash
\mathcal E(f) \subsetneq f^n\left({V_0} \cap J(f) \backslash
\mathcal E(f)\right). 
\end{equation}
Indeed, if
$\left|y-z_0\right| \leq \delta$ (\ref{eu2}) implies that there is a $y' \in W$ such that
$f^n(y') = y$ and then (\ref{eu1}) implies that $\left|y' - z_0\right|
< \delta$. Since $y' \in J(f)\backslash \mathcal E(f)$ when $y \in
J(f)\backslash \mathcal E(f)$ it follows that $\overline{V_0} \cap J(f) \backslash
\mathcal E(f) \subseteq f^n\left({V_0} \cap J(f) \backslash
\mathcal E(f)\right)$. On the other hand, it follows from (\ref{cru}) and (\ref{eu1}) that $f^n(z_1) \notin \overline{V_0}$. This shows
that (\ref{eu102}) holds. Then
$$
S = \left\{ (z,n,f^n(z)) \in \Gamma_F(n,0) : \ z \in V_0  \right\}
$$
is an open bisection in $\Gamma_F$ such that $\overline{V_0} \cap J(f) \backslash
\mathcal E(f) \subseteq s(S)$ and 
$$
\alpha_{S^{-1}} \left(\overline{V_0} \cap J(f) \backslash
\mathcal E(f)\right) \subsetneq {V_0} \cap J(f) \backslash
\mathcal E(f). 
$$
Then $V = V_0 \cap J(f) \backslash \mathcal
E(f)$ is an open subset of $U \cap J(f) \backslash \mathcal
E(f)$ such that $\alpha_{S^{-1}}(\overline{V}) \subsetneq  V$. This
shows that $\Gamma_F$ is locally contracting. 
 \end{proof}

\begin{cor}\label{simpleinf} The $C^*$-algebra
  $C^*_r\left(\Gamma_F\right)$ is simple and purely infinite.
\end{cor}
\begin{proof} By Theorem 4.16 of \cite{Th} simplicity is a consequence
  of the minimality and essential freeness of $\Gamma_F$. Pure
  infiniteness follows from Proposition 2.4 in \cite{An}
  because $\Gamma_F$ is essentially free and locally contracting.
\end{proof}

\section{The $C^*$-algebra of the exponential function}

For the statement of the next theorem, which is the main result of the
note, recall that the separable, stable, simple
purely infinite $C^*$-algebras which satisfy the universal
coefficient theorem (UCT) of Rosenberg and Schochet, \cite{RS}, is
exactly the class of $C^*$-algebras known from the Kirchberg-Phillips
results, \cite{Ph}, to be classified by their $K$-theory groups alone.

\begin{thm}\label{main}
Let $f : \mathbb C \to \mathbb C$ be an entire transcendental function such that\begin{enumerate}
\item[i)] $f'(z) \neq 0 \ \forall z \in \mathbb C$,
\item[ii)] the Julia set $J(f)$ of $f$ is $\mathbb C$, and
\item[iii)] $\# f^{-1}(f(x)) \geq 2$ for all $x \in \mathbb C$.
\end{enumerate}
Then $C^*_r\left(\Gamma_f\right)$ is the separable stable simple
purely infinite $C^*$-algebra which satisfies the UCT, and $K_0(C^*_r\left(\Gamma_f\right))
\simeq K_1(C^*_r\left(\Gamma_f\right)) \simeq \mathbb Z$.
\end{thm}    
\begin{proof} First observe that $f$ is a
  local homeomorphism because it is holomorphic with no critical points
  by assumption i). Hence $C^*_r\left(\Gamma_f\right)$ is defined. As we pointed out above the
  separability of $C^*_r\left(\Gamma_f\right)$ follows because
  $\mathbb C$ has a countable base for its topology. It follows from
  Proposition \ref{cuntzpimsner}, Lemma \ref{rep}  and Proposition 8.8 of
  \cite{Ka} that $C^*_r\left(\Gamma_f\right)$ satisfies the UCT. Since
  iii) implies that $\mathcal E(f) = \emptyset$ it follows from
  Corollary \ref{simpleinf} that $C^*_r\left(\Gamma_f\right)$ is simple and purely infinite. Since $C^*_r\left(\Gamma_f\right)$ is not
  unital (because $\mathbb C$ is not compact), it follows from Theorem
  1.2 of \cite{Z} that $C^*_r\left(\Gamma_f\right)$ is stable.

It remains now only to calculate the $K$-theory of
$C^*_r\left(\Gamma_f\right)$. We use Theorem \ref{6terms} for this and
we need therefore to determine the action on $K$-theory of the
$KK$-element $[E]$. Let $\Delta$ be a small open disc
centered at $0 \in \mathbb C$ such that $f$ is injective on
$\overline{\Delta}$. Set $V = f(\Delta)$ and let $i : C_0(\Delta) \to
C_0(\mathbb C)$ and $j : C_0(V) \to C_0(\mathbb C)$ denote the natural
embeddings. Define $\psi_f : C_0(\Delta) \to C_0(V)$ such that $\psi_f(g) =
g \circ f^{-1}$. It is easy to see that
$$
i^*[E] = j_*[\psi_f] = [j \circ \psi_f]
$$
in $KK\left(C_0(\Delta),C_0(\mathbb C)\right)$. To proceed we apply
Schoenfliess' theorem to get a
homeomorphism $F : \mathbb C\to \mathbb C$ extending $f : U \to
V$. Note that $F$ must be orientation preserving since $f$ is. It follows
therefore that $F$
is isotopic to the identity, cf. Theorem 2.4.2 on page 92 in
\cite{L}. This shows that $j \circ \psi_f$ is
homotopic to $i$ and we conclude therefore that $[j \circ \psi_f] =
[i] = i^*\left[\id_{C_0(\mathbb C)}\right]$. Since $i^* :
KK\left(C_0(\mathbb C),C_0(\mathbb C)\right)  \to
KK\left(C_0(\Delta),C_0(\mathbb C)\right)$ is an isomorphism it
follows that $[E] = \left[\id_{C_0(\mathbb C)}\right]$. The conclusion
that $K_0(C^*_r\left(\Gamma_f\right))
\simeq K_1(C^*_r\left(\Gamma_f\right)) \simeq \mathbb Z$ follows now
straightforwardly from the generalised Pimsner-Voiculescu exact
sequence of Theorem \ref{6terms}.   
\end{proof}

The function $f(z) = \lambda e^z$ clearly satisfies assumptions i) and
iii) of Theorem \ref{main} when $\lambda \neq 0$. Furthermore, when
$\lambda > \frac{1}{e}$ it is shown in \cite{De} that also assumption
ii) holds, extending the result of Misiurewics, \cite{M},  dealing
with the case $\lambda
= 1$.

\section{The $C^*$-algebra of $e^{\overline{z}}$}

Let $\mathbb K$ be the $C^*$-algebra of compact operators on an
infinite-dimensional separable Hilbert space.

\begin{thm}\label{main2}
Let $f : \mathbb C \to \mathbb C$ be an entire transcendental function such that
\begin{enumerate}
\item[i)] $f'(z) \neq 0 \ \forall z \in \mathbb C$,
\item[ii)] the Julia set $J(f)$ of $f$ is $\mathbb C$,
\item[iii)]  $\# f^{-1}(f(x)) \geq 2$ for all $x \in \mathbb C$, and
\item[iv)] $\overline{f(z)} = f\left(\overline{z}\right), \ z \in
  \mathbb C$.
\end{enumerate}
Define $\overline{f} : \mathbb C \to \mathbb C$ such that
$\overline{f}(z) = \overline{f(z)}$. Then $C^*_r\left(\Gamma_{\overline{f}}\right) \simeq \mathcal O_3
  \otimes \mathbb K$ where $\mathcal O_3$ is the Cuntz-algebra with
  $K_0(\mathcal O_3) \simeq \mathbb Z_2$ and $K_1\left(\mathcal O_3\right) =
0$, cf. \cite{C}.
\end{thm}    
\begin{proof} $C^*_r\left(\Gamma_{\overline{f}}\right)$ is separable
  and satisfies the UCT for the same reason that
  $C^*_r\left(\Gamma_{{f}}\right)$ has these properties. Since $J(f^2)
  = J(f) = \mathbb C$ and $\mathcal E(f^2) = \emptyset$ by iii) we
  conclude from Lemma \ref{julia} that $\Gamma_{f^2}$ is minimal,
  essentially free and locally contracting. Since $\Gamma_{f^2}
  \subseteq \Gamma_{\overline{f}}$ it follows that
  $\Gamma_{\overline{f}}$ is minimal and locally
  contracting. Furthermore, by using that
$$
\left\{ z \in \mathbb C : \ \overline{f}^i(z) = z \right\} \subseteq
\left\{ z \in \mathbb C : \ f^{2i}(z) = z\right\},
$$
it follows also that $\Gamma_{\overline{f}}$ is essentially free
because $\Gamma_{f^2}$ is. As in the proof of Corollary
\ref{simpleinf} we conclude now that $C^*_r\left(\Gamma_{\overline{f}}\right)$ is simple and purely
infinite. Finally, since $\overline{f}$ is orientation reversing the calculation
of the $K$-theory in the proof of Theorem \ref{main} now yields the
conclusion that $[E] = - \left[\id_{C_0(\mathbb C)}\right]$, leading
to the result that
$K_0\left(C^*_r\left(\Gamma_{\overline{f}}\right)\right) \simeq
\mathbb Z_2$ while
$K_1\left(C^*_r\left(\Gamma_{\overline{f}}\right)\right) = 0$. Hence
the theorem of Zhang, \cite{Z}, and the Kirchberg-Phillips
classification theorem, Theorem 4.2.4 of \cite{Ph}, imply that
$C^*_r\left(\Gamma_{\overline{f}}\right) \simeq \mathcal O_3 \otimes
\mathbb K$.
\end{proof}

\end{document}